\documentclass[11pt, oneside]{article}   	
\usepackage{amsmath, amsthm, amsfonts, amssymb}
\usepackage{newtxtext,newtxmath}
\usepackage[]{hyperref}
\newcommand\Ex{{\mathbb E}}

\newcommand\Prob{{\mathbb P}}

\newcommand\Normal{{\mathcal N}}

\newcommand\X{{\mathbb X}}

\newcommand\N{{\mathbb N}}

\newcommand\R{{\mathbb R}}

\newcommand\C{{\mathbb C}}

\newcommand\x{{\mathbf x}}

\newcommand\dto{\overset{d}{\to }}

\newcommand\Pto{\overset{\Prob}{\to}}

\newcommand\norm[1]{\|#1\|}

\newcommand\bra[1]{\langle #1 \rangle}
\newcommand\bralr[1]{\left\langle #1 \right\rangle}

\newtheorem{theorem}{Theorem}[section]
\newtheorem{corollary}[theorem]{Corollary}
\newtheorem{lemma}[theorem]{Lemma}

\theoremstyle{definition}

\theoremstyle{remark}
\newtheorem{remark}[theorem]{Remark}

\title{Limit theorems for moment processes of beta Dyson's Brownian motions and beta Laguerre processes}

\author{
Fumihiko Nakano\footnote{Mathematical Institute, Tohoku University, Sendai,  Japan.
\newline Email: fumihiko.nakano.e4@tohoku.ac.jp}
\and
Hoang Dung Trinh\footnote{Faculty of Mathematics Mechanics Informatics, University of Science, Vietnam National University, Hanoi, Vietnam.
\newline Email: thdung.hus@gmail.com} 
\and
Khanh Duy Trinh\footnote{Global Center for Science and Engineering, Waseda University, Japan.
\newline
Email: trinh@aoni.waseda.jp 
} 
}

\begin{document}

\maketitle
\begin{abstract}
In the regime where the parameter beta is proportional to the reciprocal of the system size, it is known that the empirical distribution of Gaussian beta ensembles (resp.\ beta Laguerre ensembles) converges weakly to a probability measure of associated Hermite polynomials (resp.\ associated Laguerre polynomials), almost surely. Gaussian fluctuations around the limit have been known as well. This paper aims to study a dynamical version of those results. More precisely, we study beta Dyson's Brownian motions and beta Laguerre processes and establish LLNs and CLTs for their moment processes in the same regime.

		
\medskip
	
\end{abstract}

\section{Introduction}
Gaussian beta ensembles, with a parameter $\beta > 0$, are one of the most studied random matrix models. They are generalizations of Gaussian orthogonal/unitary/symplectic ensembles in terms of the joint density of eigenvalues. A nice tridiagonal matrix model for them was constructed in \cite{DE02}. Based on that random matrix model, results on the limiting behavior of the empirical distribution, the edge scaling limit, the bulk scaling limit and characteristic polynomials have been established \cite{Zeitouni-2020, DE06, Duy-RIMS, Trinh-ojm-2018, Lambert-2020A, Ramirez-Rider-Virag-2011, Valko-Virag-2009}. It is worth mentioning that the joint density itself is good enough to solve problems such as the convergence to a limit, fluctuations around the limit and large deviation principles \cite{Anderson-book, Johansson98, Saff-Totik-book}.

When studying the limiting behavior of quantities of Gaussian beta ensembles as the system size $N$ tends to infinity, the parameter $\beta$ is usually assumed to be fixed. The case where $\beta$ varies as a function of $N$ has been considered recently. For example, in the regime where $\beta N \to 2c$, where $c \in (0, \infty)$ is some constant, with probability one, the empirical distribution of the eigenvalues converges weakly to a probability measure of associated Hermite polynomials interpolating a Gaussian distribution and the semi-circle distribution \cite{Allez12, Peche15, Nakano-Trinh-2018}. In the same regime, the bulk scaling limit is a homogeneous Poisson point process \cite{Peche15,Nakano-Trinh-2018, Nakano-Trinh-2020} and the edge scaling limit is the Gumbel distribution \cite{Lambert-2019}.

The aim of this paper is to study a dynamical version of the limiting behavior of the empirical distribution in the regime where $\beta N \to 2c \in (0, \infty)$. Let us first introduce related results. We concern with (scaled) Gaussian beta ensembles of ordered eigenvalues whose joint density is of the form
\begin{equation}\label{GbE}
	\frac1{Z_{N, \beta}}\prod_{1\le i<j \le N} |\lambda_j - \lambda_i|^{\beta} \prod_{l=1}^N e^{-\frac{\lambda_l^2}2}, \quad (\lambda_1 \le  \dots \le \lambda_N),
\end{equation}
where $Z_{N, \beta}$ is  the normalizing constant. Let 
\[
	L_N = \frac{1}{N} \sum_{i=1}^N \delta_{\lambda_i},
\]
be the empirical distribution, where $\delta_\lambda$ denotes the Dirac measure. Then as $N \to \infty$ with $\beta N \to 2c \in (0, \infty)$, the sequence $\{L_N\}$ converges weakly to a probability measure $\nu_c$, almost surely \cite{Peche15, Nakano-Trinh-2018, Trinh-2019}. The density of $\nu_c$, for $c > 0$, is given by 
\begin{equation}\label{nuc}
	\nu_c (x) = \frac{e^{-x^2/2}}{\sqrt {2\pi} |\hat f_c(x)|^2} , \quad \hat f_c(x) = \sqrt{\frac c{\Gamma(c)}}  \int_0^\infty   t^{c-1} e^{-t^2/2} e^{ixt} dt .
\end{equation}
We remark here that the family $\nu_c$, for $c > -1$, appeared in \cite{Askey-Wimp-1984} as probability measures of associated Hermite polynomials. The convergence means that for any bounded continuous function $f$, 
\[
	\bra{L_N, f} = \frac1N \sum_{i=1}^N f(\lambda_i)\to \bra{\nu_c, f} \quad \text{as} \quad N \to \infty, \quad \text{almost surely.}
\]
Here $\bra{\mu, f}$ denotes the integral $\int f(x) d\mu(x)$ for a probability measure $\mu$ and an integrable function $f$. The above still holds when $f$ is a continuous function with some growth condition at infinity (e.g.\ polynomial growth). We refer to those convergences as the law of large numbers (LLN).
Fluctuations around the limit were already established \cite{Nakano-Trinh-2018, Trinh-2019}. For a continuously differentiable function $f$ whose derivative $f'$ is of polynomial growth, the following central limit theorem (CLT) holds. As $N \to \infty$, 
\begin{equation}\label{CLT-f}
	\sqrt{N} \Big(\bra{L_N, f} - \Ex[\bra{L_N, f}] \Big) \dto \Normal(0, \sigma_{f,c}^2),
\end{equation}
where $\sigma_{f, c}^2 \ge 0$ is the limiting variance and `$\dto$' denotes the convergence in distribution. CLTs are first established for polynomial test functions by a martingale approach based on the tridiagonal random matrix model. Then extending from polynomials to such differentiable functions is done by a standard method in random matrix theory. The limiting variance $\sigma_{f,c}^2$ when the test function $f(x)=x^n$ has been calculated by using stochastic analysis \cite{Spohn-2020} and loop equations \cite{Forrester-2021}.

Our object in this paper is the so-called beta Dyson's Brownian motions defined to be the strong solution of the following system of stochastic differential equations (SDEs)
\begin{equation}\label{Dyson}
\begin{cases}
	d\lambda_i(t)= d  b_i(t)  + \dfrac\beta2 \sum\limits_{j : j \neq i} \dfrac{1}{\lambda_i(t) - \lambda_j(t)} dt,\\
	\lambda_i(0) = 0,
\end{cases}
 i = 1, \dots, N,
\end{equation}
together with the constraint that 
$
(\lambda_1(t), \lambda_2(t), \dots, \lambda_N(t)) \in \Omega_G, (t > 0).
$
Here $\{b_i(t)\}_{i=1}^N$ are independent standard Brownian motions and 
\[
	\Omega_G = \{(\lambda_1,  \dots, \lambda_N) \in \R^N : \lambda_1 \le \cdots \le \lambda_N\}.
\]
For $\beta \ge 1$, the above system of SDEs has a unique strong solution and $\lambda_1(t), \dots, \lambda_N(t)$ never collide for $t > 0$ (see \cite[\S 4.3]{Anderson-book}, for example). Dyson~\cite{Dyson-1962} showed that the eigenvalue process of  symmetric (resp.\ Hermitian) matrix-valued processes of independent standard real (resp.\ complex) Brownian motions solve the above SDEs with $\beta = 1$ (resp.\ $\beta = 2$).
When $0< \beta <1$, the above SDEs were treated in \cite{Cepa-Lepingle-1997} by using multivalued SDEs. 
In this case, almost surely, the set of $t$ such that $\lambda_i(t) = \lambda_j(t)$ for some $i \neq j$ has Lebesgue measure zero. The process $\{\lambda_i(t)\}$ can also be explained as a type A radial Dunkl process, a Markov process on $\Omega_G$ with generator
\[
	L[f](\x) = \frac12 \sum_{i=1}^N \frac{\partial^2 f}{\partial x_i^2}(\x) + \frac\beta 2 \sum_{j < i} \frac{1}{x_i - x_j} \left( \frac{\partial f}{\partial x_i} -  \frac{\partial f}{\partial x_j}\right)(\x),  \quad \x = (x_1, \dots, x_N) \in \Omega_G,
\]
for suitable $f \in C^2(\Omega_G)$. It turns out that (see \cite{Rosler-Voit-1998}) for $t>0$, $\{\lambda_i(t) /\sqrt t\}$ is distributed according to the Gaussian beta ensemble~\eqref{GbE}.

We study the limiting behavior of the empirical measure process 
\[
	\mu_t^{(N)} = \frac1N \sum_{i=1}^N \delta_{\lambda_i(t)},
\]
through its moment processes 
\[
	S_n^{(N)} (t) = \bra{\mu_t^{(N)}, x^n} = \frac{1}{N}\sum_{i=1}^N \lambda_i(t)^n, \quad (n=0,1,2,\dots).
\]
For fixed $T>0$, regard $S_n^{(N)}$ as random elements on the space $C([0, T])$ of continuous functions on $[0, T]$ endowed with the supremum norm, we establish the LLN and CLT for each moment process in the regime where $\beta N \to 2c$. Precise statements are stated in Theorems~\ref{thm:Gaussian-LLN},~\ref{thm:joint-moment} and \ref{thm:Gaussian-CLT-orthogonal}. Since the Gaussian beta ensemble~\eqref{GbE} is the distribution of $\{\lambda_i(t)\}$ at time $t = 1$, results on the process level imply the following CLT.
\begin{theorem}\label{thm:main-intro}
	Let $\{p_n\}$ be associated Hermite polynomials (orthogonal with respect to $\nu_c$) which are defined recursively as 
\begin{align*}
	&p_0 = 1, \quad p_1 = xp_0 = x, \\
	&p_{n+1} = xp_n - (n+c)p_{n-1}, \quad (n \ge 1).
\end{align*}
Let $P_n$ be a primitive of $p_n$, that is, $P_n' = p_n$. Then as $N \to \infty$ with $\beta N = 2c$, 
\[
	\sqrt N \Big(\bra{L_N, P_n} - \bra{\nu_c, P_n} \Big) \dto \Normal(0, \sigma_{P_n, c}^2), 
\]
where 
\[
	\sigma_{P_n, c}^2  = \frac{1}{n+1} \int p_n(x) p_n(x) d\nu_c(x) = \frac{1}{n+1} (c+1) \cdots (c+n).
\]
Moreover, for any $M\in \N$, $\{\sqrt N (\bra{L_N, P_n} - \bra{\nu_c, P_n} )\}_{n=0}^M$ jointly converge in distribution to independent Gaussian random variables. 
\end{theorem}

The above theorem is included in the statement of Theorem~\ref{thm:Gaussian-CLT-orthogonal}.
This result provides more information than the CLT~\eqref{CLT-f} mentioned above. Such type of CLT statements concerning with orthogonal polynomials has already been known for the three classical beta ensembles (Gaussian beta ensembles, beta Laguerre ensembles and beta Jacobi ensembles) with fixed $\beta$ \cite{Cabanal-French, Dumitriu-Paquette-2012}. This paper uses an idea of choosing orthogonal polynomials from \cite{Cabanal-French} but detailed arguments are totally different.
We establish analogous results on the Laguerre case as well. The Jacobi case which requires more works will be written in a forthcoming paper.

The paper is organized as follows. LLNs and CLTs for moment processes of the eigenvalue process \eqref{Dyson} (the Gaussian case) are established in Section 2. Then Section 3 deals with the Laguerre case by using similar arguments.

\section{Gaussian case}
In this section, we study the limiting behavior of the empirical measure process $\mu_t^{(N)}$ of beta Dyson's Brownian motions~\eqref{Dyson} when $\beta = 2c/N$, where  $c \in (0, \infty)$ is a given positive constant. In case  $\beta$ is fixed, the LLN and CLT for the empirical measure process of more general eigenvalue processes related to beta ensembles have been studied \cite{Rogers-Shi-1993,Unterberger-2018,Unterberger-2019}. In our considering regime, the LLN for $\mu_t^{(N)}$ is a special case of general results in \cite{Cepa-Lepingle-1997}. However, we are going to re-prove the LLN for $\mu_t^{(N)}$ by a moment method  developed in \cite{Trinh-Trinh-BLP} in the study of beta Laguerre processes.

For $f = f(t, x) \in C^2([0, \infty) \times \R)$, we write $\partial_t f = \partial f / \partial t, f' = \partial f / \partial x$ and $f''= \partial^2 f / \partial x^2$. Then  by It\^o's formula,
\begin{align}
	d\bra{\mu^{(N)}_t, f} &= \frac 1N \sum_{i=1}^N df(t, \lambda_i(t))\notag\\
	&= \frac1N\sum_{i=1}^N f'(t, \lambda_i(t)) d\lambda_i(t) + \frac1{2N} \sum_{i=1}^N f''(t, \lambda_i(t)) dt + \frac1N \sum_{i=1}^N\partial_t f(t, \lambda_i(t)) dt\notag\\
	&=\frac1N\sum_{i=1}^N f'(t, \lambda_i(t))  d b_i(t) + \frac{c}{2N^2}\sum_{i=1}^N \sum_{j \neq i} \frac{f'(t, \lambda_i(t))}{\lambda_i(t) - \lambda_j(t)} dt\notag\\
	&\quad+ \bra{\mu_t^{(N)}, \frac12 f'' + \partial_t f} dt\notag\\
	&=\frac{1}N\sum_{i=1}^N f'(t, \lambda_i(t))  d b_i(t)  + \frac c2 \iint \frac{f'(t, x) - f'(t, y)}{x - y}d\mu_t^{(N)}(x) d\mu_t^{(N)}(y)dt \notag\\
	&\quad + \bra{\mu_t^{(N)}, \frac12 f'' + \partial_t f} dt  - \frac{c}{2N} \bra{\mu_t^{(N)}, f''}dt. \label{Ito-f}
\end{align}
To be more precise, the above formula holds when $\lambda_1(t), \dots, \lambda_N(t)$ are all distinct, which occurs almost surely for almost every $t \in \R$. The form~\eqref{Ito-f} is a starting point to deal with many problems in the study of stochastic processes related to beta ensembles.

\subsection{Law of large numbers}
Denote by $S_n^{(N)}(t) = \bra{\mu_t^{(N)}, x^n} = \frac1N \sum_{i=1}^N \lambda_i(t)^n$ the $n$th moment process of $\{\lambda_i(t)\}$.
The formula~\eqref{Ito-f} for $f = x^n$ reads 
\begin{align*}
	dS_n^{(N)}(t) &=  	\frac{ n}N\sum_{i=1}^N \lambda_i(t)^{n - 1}  d b_i(t) + \frac{c n}2 \sum_{j = 0}^{n-2} S_{j}^{(N)}(t) S_{n-2-j}^{(N)}(t) dt \\
	&\quad + \frac12 n(n-1) S_{n-2}^{(N)}(t) dt - \frac c{2N} n(n-1) S_{n-2}^{(N)}(t) dt,
\end{align*}
or in the integral form 
\begin{align}
	S_n^{(N)}(t) &=  	\frac{ n}N\sum_{i=1}^N \int_0^t \lambda_i(u)^{n - 1}  d b_i(u) + \frac{c n}2 \sum_{j = 0}^{n-2} \int_0^t  S_{j}^{(N)}(u) S_{n-2-j}^{(N)}(u) du \notag \\
	&\quad + \frac12 n(n-1) \int_0^t S_{n-2}^{(N)}(u) du - \frac c{2N} n(n-1) \int_0^t S_{n-2}^{(N)}(u) du, \quad (n \ge 1).	\label{Gauss-Sn}
\end{align}

Let $T>0$ be fixed and let $C([0, T])$ be the space of continuous functions on $[0,T]$ endowed with the supremum norm $\norm{\cdot}_\infty$. We regard $S^{(N)}_n$ as a $C([0, T])$-valued random element. The idea now is to use induction to show the convergence of the $n$th moment process $S_n^{(N)}$  in the space $C([0, T])$. This idea has been used in \cite{Trinh-Trinh-BLP} to establish the LLN for the empirical process of beta Laguerre processes. A sequence $\{X^{(N)}\}_N$ of $C([0, T])$-valued random elements is said to converge in probability to a non-random element $x \in C([0, T])$ if for any $\varepsilon > 0$, 
\[
		\lim_{N \to \infty} \Prob(\norm{X^{(N)} - x}_\infty \ge \varepsilon) =0.
\] 
We denote the convergence in probability by `$\Pto$'.
Our result is as follows.
\begin{theorem}\label{thm:Gaussian-LLN}
There is a probability measure-valued process $(\mu_t)_{0 \le t \le T}$ such that  for any polynomial $p(x)$, as $N \to \infty$,
	\[
		\bra{\mu_t^{(N)}, p} \Pto \bra{\mu_t, p}.
	\]
The convergence still holds when $f(t, x)$ is a polynomial in $t$ and $x$, that is,  as $N \to \infty$,
\[
	\bra{\mu_t^{(N)}, f} \Pto \bra{\mu_t, f}.
\]
Here for a function $f(t, x)$ of two variables $t$ and $x$, the integral $\bra{\mu_t^{(N)}, f}$ is taken over $x$. Moreover, $\bra{\mu_t, f}$ is differentiable (as a function of $t$) and the following relation holds 
\begin{equation}\label{integral-equation}
	\partial_t \bra{\mu_t, f} = \frac c2 \iint \frac{f'(t, x) - f'(t, y)}{x - y}d\mu_t(x) d\mu_t(y)  + \bra{\mu_t, \frac12 f'' + \partial_t f}.
\end{equation}
Here note that we also use the partial derivation notation $\partial_t \bra{\mu_t, f}$ to denote the derivative with respect to $t$, although the function $\bra{\mu_t, f}$ depends only on $t$.
\end{theorem}
\begin{remark}
The limiting process $(\mu_t)_{0\le t \le T}$ will be identified later in the proof of Theorem~\ref{thm:Gaussian-LLN}. It turns out that the probability measure $\mu_t$ is determined by moments. Thus, the convergence of all moment processes implies that the sequence of empirical measure processes $(\mu_t^{(N)})_{0 \le t \le T}$ converges to $(\mu_t)_{0 \le t \le T}$ in probability in the space of continuous probability measure-valued processes endowed with the topology of uniform convergence. Refer to \cite{Trinh-Trinh-BLP} for the derivation. 
\end{remark}

We prove Theorem~\ref{thm:Gaussian-LLN} through several lemmas. We begin with the martingale part.
\begin{lemma}
For each $n=1,2,\dots$, denote by
\[
	M_n^{(N)}(t) = \frac{ n}N\sum_{i=1}^N \int_0^t \lambda_i(u)^{n - 1}  d b_i(u)
\]
the martingale part in $S_n^{(N)}$. Then 
\[
	M_n^{(N)}\Pto 0 \quad \text{as} \quad N \to \infty.
\]
\end{lemma}
\begin{proof}
Note that $\int_0^t \Ex[\sum_i \lambda_i(u)^{2k}] du < \infty$, for any $k = 1, 2,\dots$,  because we know exactly the joint distribution of $\{\lambda_i(u)\}$. Thus, $M_n^{(N)}$ is a martingale with the quadratic variation 
\[
	[M_n^{(N)}](t) = \frac{n^2}{N} \int_0^t \frac{\sum_{i=1}^N \lambda_i(u)^{2n-2}}{N} du = \frac{n^2}{N} \int_0^t \bra{\mu_u^{(N)}, x^{2(n-1)}} du.
\]
Using Doob's martingale inequality, we deduce that
\begin{align*}
	\Prob \left(\norm{M_n^{(N)}}_\infty \ge \varepsilon \right) &= \Prob \left(\sup_{0 \le t \le T} |M_n^{(N)}(t)| \ge \varepsilon \right) \\
	&\le \frac{\Ex[M_n^{(N)}(T)^2]}{\varepsilon^2} = \frac{n^2}{\varepsilon^2 N}\int_0^T \Ex[\bra{\mu_t^{(N)}, x^{2(n-1)}}] dt.
\end{align*}
Recall that $\{\lambda_i(t) / \sqrt t\}$ is distributed as the Gaussian beta ensemble with parameter $\beta = 2c/N$. It follows that $\Ex[\bra{\mu_t^{(N)}, x^{2(n-1)}}] = t^{n-1}\Ex[\bra{L_N, x^{2(n-1)}}]$ with $L_N$ being  the empirical distribution defined in the introduction. In this regime, the sequence $\{\Ex[\bra{L_N, x^{2(n-1)}}]\}_N$ is easily shown to be bounded (for example, see \cite{Nakano-Trinh-2018, Nakano-Trinh-2020}). Then the integral in the above estimate is bounded (as $N \to \infty$), implying the desired result. The lemma is proved.
\end{proof}

Next, we introduce the following fundamental results whose proof is standard and thus is omitted. They will be used to derive the convergence of $S_n^{(N)}$ by induction.
\begin{lemma}\label{lem:additivity}
Let $\{X^{(N)}\}_N$ and $\{Y^{(N)}\}_N$ be two sequences of $C([0, T])$-valued random elements.  
Assume that $X^{(N)}$ and $Y^{(N)}$ converge in probability to non-random limits $x$ and $y$, respectively. Then the following hold.
\begin{itemize}
	\item[\rm (i)] As $N\to \infty$, 
	\[
		a(t) X^{(N)}(t) + b(t) Y^{(N)}(t) \Pto a(t) x(t) + b(t) y(t) , 
	\]
	where $a, b \in C([0,T])$.	
	\item[\rm (ii)] As $N\to \infty$,  
	\[
		X^{(N)}(t)  Y^{(N)}(t) \Pto x(t) y(t).
	\]
	
	\item[\rm (iii)] As $N\to \infty$, 
	\[
		\int_0^t X^{(N)}(u)du \Pto \int_0^t x(u) du.
	\]
\end{itemize}
\end{lemma}

\begin{proof}[Proof of Theorem~\rm\ref{thm:Gaussian-LLN}]
For $n=0,1$, it is clear that
\[
	S_0^{(N)}(t) \equiv 1,\quad 
	S_1^{(N)}(t) = M_1^{(N)}(t) \Pto 0.
\]
Now we show by induction that each moment process converges in probability to a non-random limit in $C([0, T])$ as $N\to \infty$. Indeed, assume that for any $l \le n-1$, the $l$th moment process $S_{l}^{(N)}$ converges in probability to $m_l(t) \in C([0, T])$ as $N \to \infty$. We need to show that $S_n^{(N)}$ also converges. Observe that the martingale part and the last term in the equation~\eqref{Gauss-Sn} converge to zero. The remaining terms converge by induction assumption with the help of Lemma~\ref{lem:additivity}. Combining all those observations, we deduce that 
\[
	S_{n}^{(N)}(t) \Pto m_n(t) \quad \text{as} \quad N \to \infty, 
\]
where 
$m_n(t)$ is defined as 
\begin{equation}\label{mn}
	m_n(t) =\int_0^t\bigg( \frac12 n(n - 1)m_{n-2}(u) + \frac{cn}2 \sum_{j=0}^{n-2} m_j(u) m_{n-2-j}(u) \bigg) du.
\end{equation}

Note that $m_1(t) \equiv 0$. It then follows that the limiting moment process $m_n(t)$ is zero, $m_n(t) \equiv 0$, for any odd $n$. By induction, the equation \eqref{mn} implies first that $m_n(t)$ is a continuous function, for any $n \ge 0$. Then taking the derivative of the integral, we obtain a differential equation
\begin{align}
	\partial_t m_n(t) &=\frac12 n(n - 1)m_{n-2}(t) + \frac{cn}2 \sum_{j=0}^{n-2} m_j(t) m_{n-2-j}(t) \notag\\
	&=\frac12 n(n - 1)m_{n-2}(t) + \frac{cn}2 \sum_{j=0, j :\text{even}}^{n-2} m_j(t) m_{n-2-j}(t) \label{momentn}
\end{align}
with initial condition $m_n(0)=0\ (n \ge 1)$.
By induction again, we deduce that $m_n(t)$ is of the form 
\[
	m_n(t) = u_n t^{n/2},
\]
where $\{u_n\}_{n \ge 0}$ is a sequence of real numbers satisfying a self-convolutive recurrence
\begin{equation}\label{moment-recurrence}
\begin{cases}
	u_{2n} = (2n - 1) u_{2n-2} + c \sum_{j = 0, j:\text{even}}^{2n-2} u_{j} u_{2n-2-j},\quad (u_0 = 1),\\
	u_{2n+1} = 0.
\end{cases}
\end{equation}
It turns out that $\{u_{n}\}$ are moments of the probability measure $\nu_c$ of associated Hermite polynomials \cite{DS15},
\[
	u_n = \int_\R x^n \nu_c(x) dx.
\]
The probability measure $\nu_c$ is determined by moments and its density $\nu_c(x)$ is given in \eqref{nuc}. 

Next, we define the limiting probability measure-valued process $(\mu_t)$ as
\[
	\mu_t(dx) = \frac1{\sqrt t}\nu_c(x / \sqrt t) dx, \quad(t >0), \quad \mu_0 = \delta_0.
\]
Then 
\[
	\bra{\mu_t, x^n} =  u_{n} t^{n/2} = m_n(t). 
\]
Now we rewrite the convergence of moment processes in the following form: for $n=0,1,\dots,$ as $N\to \infty$,
\[
	\bra{\mu_t^{(N)}, x^n} \Pto \bra{\mu_t, x^n},
\]
\[
	\partial_t \bra{\mu_t, x^n} = \frac12  \bra{\mu_t,n(n-1) x^{n-2}} + \frac{c}{2}\iint\frac{nx^{n-1} - n y^{n-1}}{x-y} d\mu_t(x) d\mu_t(y). 
\]
Lemma~\ref{lem:additivity} implies that the convergence holds for any polynomial $p(x)$ and any polynomial $f(t,x)$. The equation~\eqref{integral-equation} follows immediately from the equation~\eqref{Ito-f}. Theorem~\ref{thm:Gaussian-LLN} is proved.
\end{proof}

\begin{remark}
As the limiting measure of the empirical distribution of the Gaussian beta ensemble~\eqref{GbE} in the  regime $\beta N \to 2 c \in (0, \infty)$,
the measure $\nu_c$ in the proof of Theorem~\ref{thm:Gaussian-LLN} was calculated explicitly in \cite{Allez12, DS15}. It is nothing but the probability measure of associated Hermite polynomials (see \cite{Askey-Wimp-1984} or \cite[\S 5.6]{Ismail-book}).
	Theorem~\ref{thm:Gaussian-LLN} not only shows the convergence at the process level but also provides another way to identify the limiting measure $\nu_c$. Recall that fluctuations around the limit were established in \cite{Nakano-Trinh-2018, Trinh-2019} by using the tridiagonal random matrix model. It is known that for a continuously differentiable function $f$ whose derivative $f'$ is of polynomial growth, the following central limit theorem holds. As $N \to \infty$, 
\begin{equation*}
	\sqrt{N} \Big(\bra{L_N, f} - \Ex[\bra{L_N, f}] \Big) \dto \Normal(0, \sigma_{f,c}^2),
\end{equation*}
where $\sigma_{f, c}^2 \ge 0$ is the limiting variance. In the next subsection, we are going to derive CLTs at the process level. As a consequence, we recover the above CLT when $f$ is a polynomial with more information on the limiting variance.

\end{remark}

\subsection{Central limit theorem}
We begin this section by introducing several concepts on the convergence in distribution of continuous processes.

For $d \ge 1$, let $C([0, T]; \R^d)$ denote the space of continuous functions $f \colon [0, T] \to \R^d$ endowed with the supremum norm. (When $d = 1$, we have used $C([0, T]])$ instead of $C([0, T]; \R^d)$.) Since $C([0, T]; \R^d)$ is a complete, separable metric space, the weak convergence, or the convergence in distribution of $C([0, T]; \R^d)$-valued random elements are defined as usual as follows. A sequence $X^{(N)}$ converges in distribution/weakly to $X$ if for any bounded continuous function $F \colon C([0, T]; \R^d) \to \R$, 
\[
	\Ex[F(X^{(N)})] \to \Ex[F(X)] \quad \text{as} \quad N \to \infty.
\]
We use the same notation `$\dto$' to denote that convergence. The joint convergence of $C([0, T])$-valued random elements will be understood as the convergence of $C([0, T]; \R^d)$-valued random elements. Namely, we say that $X^{(N; 1)}, X^{(N; 2)}, \dots, X^{(N; d)}$ jointly converge in distribution to $X_1, \dots, X_d$ (as random elements on $C([0, T])$) if $X^{(N)}=(X^{(N; 1)}, X^{(N; 2)}, \dots, X^{(N; d)})$ converges in distribution to $X=(X_1, \dots, X_d)$ as random elements on  $C([0, T]; \R^d)$. 

The continuous mapping theorem implies the convergence of finite dimensional distributions, that is, 
for any $t_1, \dots, t_k \in [0, T]$, 
\[
	(X^{(N)}(t_1), X^{(N)}(t_2),\dots, X^{(N)}(t_k)) \dto (X(t_1), X(t_2), \dots, X(t_k)) \text{ as } N \to \infty, 
\]
as the convergence in distribution of $\R^{kd}$-valued random elements. The converse is not true in general. It turns out that the sequence $\{X^{(N)}\}_N$ converges in distribution to $X$, if and only if (i) the sequence $\{X^{(N)}\}_N$ is tight, and (ii) $\{X^{(N)}\}_N$  converges in finite dimensional distributions to $X$.

We now deal with the convergence in distribution of the martingale part in \eqref{Ito-f}.
\begin{lemma}\label{lem:joint}
Let $f_k$ be a polynomial in $t$ and $x$, for $k = 1, \dots, n$. 
Define  
\[
	\Phi^{(f_k; N)}(t) = \frac{1}{\sqrt N} \sum_{i=1}^N \int_0^t f_k(u, \lambda_i(u)) db_i(u).
\]
Then $\{\Phi^{(f_k; N)} \}_{k=1}^n$ jointly converge to Gaussian processes $\{\eta_k\}_{k=1}^n$ of mean zero and covariance 
\[
	\Ex[\eta_k(s) \eta_l(t)] = \int_0^{s\wedge t} \bra{\mu_u, f_k(u, x) f_l(u, x)} du.
\]

\end{lemma}
\begin{proof}
Note that $\{\Phi^{(f_k; N)}\}_{k=1}^n$ are martingales with the following cross-variation 
\begin{align*}
	[\Phi^{(f_k; N)},  \Phi^{(f_l; N)}](t) &= \int_0^t \frac{ \sum_{i=1}^N f_k(u, \lambda_i(u)) f_l(u, \lambda_i(u))}{N} du \\
	&= \int_0^t \bra{\mu_u^{(N)}, f_k(u, x) f_l(u,x)} du.
\end{align*}
Since the product $f_k f_l$ is again a polynomial, the LLN in Theorem~\ref{thm:Gaussian-LLN} implies that 
\[
	[\Phi^{(f_k; N)}, \Phi^{(f_l; N)}](t) \to \int_0^t \bra{\mu_u, f_k(u, x) f_l(u, x)} du\quad \text{as} \quad N \to \infty.
\]
For each $t$, the convergence holds in probability and in $L^q$, for any $q \in [1, \infty)$. Then by a general theorem in \cite{Rebolledo-1980}, there are Gaussian processes $\{\eta_k\}_{k=1}^n$ of mean zero and covariance 
\[
	\Ex[\eta_k(s) \eta_l(t)] = \int_0^{s\wedge t} \bra{\mu_u, f_k(u, x) f_l(u, x)} du,
\]
to which $\{\Phi^{(f_k; N)}\}_{k=1}^n$ jointly converge. The lemma is proved.
\end{proof}

Now let us study the moment processes. Let
\[
	\tilde S_n^{(N)}(t) = \sqrt N \Big(S_n^{(N)}(t) - m_n(t) \Big).
\]
It follows from the formula~\eqref{Gauss-Sn} and the relation~\eqref{mn} that
\begin{align}
	\tilde S_n^{(N)}(t) &=  	\frac{n}{\sqrt N}\sum_{i=1}^N \int_0^t \lambda_i(u)^{n - 1}  d b_i(u) \notag \\
	&\quad + \frac{c n}2 \sum_{j = 0}^{n-2} \int_0^t  \sqrt N \Big( S_{j}^{(N)}(u) S_{n-2-j}^{(N)}(u) - m_j(u) m_{n-2-j}(u)\Big) du \notag \\
	&\quad + \frac12 n(n-1) \int_0^t \tilde S_{n-2}^{(N)}(u) du - \frac c{2\sqrt N} n(n-1) \int_0^t S_{n-2}^{(N)}(u) du. \label{Sntilde}
\end{align}
Note that the martingale part in the above formula is identical with $\Phi^{(n x^{n-1}; N)}$. Thus, Lemma~\ref{lem:joint} implies that $\{\Phi^{(m x^{m-1}; N)} \}_{m=1}^n$ jointly converge in distribution to Gaussian processes $\{\eta_m \}_{m=1}^n$. The following result is analogous to Theorem 4.3.20 in \cite{Anderson-book} in case $\beta$ is fixed.

\begin{theorem}\label{thm:joint-moment}
	For each $n$, 
	\[
		\tilde S_n^{(N)}(t) = \sqrt N \left(S_n^{(N)}(t) - m_n(t) \right) \dto \xi_n,
	\]
where $\xi_n$ is a Gaussian process defined inductively as 
\begin{equation}\label{xin}
	\xi_n(t) = \eta_n(t) + cn \sum_{j=0}^{n-2} \int_0^t m_j(u) \xi_{n-2-j}(u) du + \frac12 n(n-1) \int_0^t \xi_{n-2}(u)du.
\end{equation}
The joint convergence also holds.
\end{theorem}

We need the following lemma.
\begin{lemma}\label{lem:product}
	Assume that $\tilde S_{n_1}^{(N)}$ and $\tilde S_{n_2}^{(N)}$ jointly converge in distribution to $\xi_{n_1}$ and  $\xi_{n_2}$. Then
\begin{equation}
	\sqrt{N}\Big( S_{n_1}^{(N)}(t) S_{n_2}^{(N)}(t) - m_{n_1}(t) m_{n_2}(t) \Big)  - \Big(m_{n_1}(t) \tilde S_{n_2}(t)  + m_{n_2}(t)\tilde S_{n_1}(t) \Big) \Pto 0.\label{Ptozero}
\end{equation}
Consequently, 
\begin{equation}\label{Sn1n2}
	\sqrt{N}\Big( S_{n_1}^{(N)}(t) S_{n_2}^{(N)}(t) - m_{n_1}(t) m_{n_2}(t) \Big)  
	 \dto m_{n_1}(t) \xi_{n_2}(t) + m_{n_2}(t)\xi_{n_1}(t).
\end{equation}
\end{lemma}
\begin{proof}
It is clear that 
\begin{align*}
	&\sqrt{N}\Big( S_{n_1}^{(N)}(t) S_{n_2}^{(N)}(t) - m_{n_1}(t) m_{n_2}(t) \Big) - \Big(m_{n_1}(t)\tilde S_{n_2}(t) +  m_{n_2}(t) \tilde S_{n_1}(t) \Big)\\
	&= (S_{n_1} (t) - m_{n_1}(t) )  \tilde S_{n_2}(t)  \Pto 0.
\end{align*}
Now as a consequence of the continuous mapping theorem, the joint convergence of $\tilde S_{n_1}$ and $\tilde S_{n_2}$ implies that 
\begin{equation}\label{jointSn1Sn2}
	m_{n_1}(t) \tilde S_{n_2}(t)  + m_{n_2}(t)\tilde S_{n_1}(t) \dto m_{n_1}(t) \xi_{n_2}(t) + m_{n_2}(t)\xi_{n_1}(t).
\end{equation}
The convergence~\eqref{Sn1n2} follows from the two equations~\eqref{Ptozero} and \eqref{jointSn1Sn2} by a general result (Theorem~3.1 in \cite{Billingsley-book-1999}).
The lemma is proved.
\end{proof}

\begin{proof}[Proof of Theorem~{\rm\ref{thm:joint-moment}}]
	Let $M \in \N$ be fixed. We aim to show that 
\[
	\{\tilde S_m^{(N)}\}_{m=1}^M \text{ jointly converge to } \{\xi_m\}_{m=1}^M,
\]
with $\xi_n(t)$ defined in \eqref{xin}. Let us start from the joint convergence 
\[
	\{\Phi^{(mx^{m-1}; N)}\}_{m=1}^M \dto \{\eta_m\}_{m=1}^M.
\]
Since $\tilde S_1^{(N)} = \Phi^{(1x^{0}; N)}$, we replace the last entry in the above vector of $\Phi$'s by $\tilde S_1^{(N)}$ without changing the vector of $\eta$'s ($\xi_1 = \eta_1$). Assume for induction that for $n < M$, the following joint convergence holds
\begin{equation}\label{upton}
	\X^{(N)}:= \left\{\{\Phi^{(mx^{m-1}; N)}\}_{m=n}^M, \{\tilde S_m^{(N)}\}_{m=1}^{n-1} \right\} \dto \left\{\{\eta_m\}_{m=n}^M, \{\xi_m\}_{m=1}^{n-1}\right\}.
\end{equation}
Our task is to show the above convergence with $n$ replaced by $n+1$. 
Define the function $F \colon C([0, T]; \R^M) \ni \{x^{(m)}\}_{m=1}^M \to C([0, T])$ by 
\begin{align*}
	F(\{x^{(m)}\}_{m=1}^M)(t) &= x^{(n)}(t)  + cn \sum_{j=0}^{n-2} \int_0^t m_j(u) x^{(n-2-j)}(u) du \\
	&\quad + \frac12 n(n-1) \int_0^t x^{(n-2)}(u)du.
\end{align*}
It is clear that $F$ is a continuous function. Thus, by the continuous mapping theorem, 
\[
	\left\{\{\Phi^{(mx^{m-1}; N)}\}_{m=n+1}^N, F(\X^{(N)}), \{\tilde S_m^{(N)}\}_{m=1}^{n-1} \right\} \dto \left\{\{\eta_m\}_{m=n+1}^M, \{\xi_m\}_{m=1}^{n}\right\}.
\]
Now the formula~\eqref{Sntilde} for $\tilde S_n^{(N)}$ together with Lemma~\ref{lem:product} implies that 
\[
	F(\X^{(N)}) - \tilde S_n^{(N)} \Pto 0.
\]
Again by a general theorem (Theorem 3.1 in \cite{Billingsley-book-1999}), it follows that $F(\X^{(N)})$ can be replaced by $\tilde S_n^{(N)}$. The proof is complete.
\end{proof}

The joint convergence in Theorem~\ref{thm:joint-moment} implies that for any polynomial $f = f(t, x) $ in $t$ and $x$,
\[
	\sqrt N \left(\bra{\mu^{(N)}_t, f } - \bra{\mu_t, f} \right) \dto \bra{\xi, f},
\]
where the limiting process, denoted by $\bra{\xi, f}$, is defined to be a finite linear combination of $\{\xi_n\}$. In particular, the following linearity holds
\[
	\bralr{\xi, \sum_{finite} a_f f} = \sum_{finite} a_f \bra{\xi, f}, \quad a_f \in \R.
\]
Using such notations, we can rewrite the statement in Lemma~\ref{lem:product} as
\begin{align*}
		& \iint (x^{n_1} y^{n_2} + x^{n_2} y^{n_1})  \Big[\sqrt N d\mu_t^{(N)}(x) d\mu_t^{(N)}(y) - \sqrt N d\mu_t(x) d\mu_t(y)\Big]\\
		&\quad  \dto \bralr{\xi, 2 \int (x^{n_1} y^{n_2} + x^{n_2} y^{n_1})  d\mu_t(y)} .
\end{align*}
Here and in what follows, $\xi$ is assumed to act upon functions of $t$ and $x$ variables.

Now let $f=f(t,x)$ be a polynomial in $t$ and $x$. Since  
\[
	\frac{f'(t, x) - f'(t, y)}{x - y} 
\] 
is a linear combination of $t^k (x^{n_1} y^{n_2} + x^{n_2} y^{n_1})$, it follows that 
\begin{align*}
	&\iint \frac{f'(t, x) - f'(t, y)}{x - y}  \Big[\sqrt N d\mu_t^{(N)}(x) d\mu_t^{(N)}(y) - \sqrt Nd\mu_t(x) d\mu_t(y)\Big] \\
	&\quad \dto \bralr{\xi, 2\int \frac{f'(t, x) - f'(t, y)}{x - y} \mu_t(dy)}.
\end{align*}
Consequently, we get the following result.

\begin{corollary}\label{lem:polytx}
Let $f=f(t,x)$ be a polynomial in $t$ and $x$. Denote by $\bra{\eta, f}$ the limit of the martingale part in the expression of $\bra{\mu_t^{(N)}, f}$, or the limit of $\Phi^{(f'; N)}$ which is coupling with $\{\eta_m\}_m$ such that the joint convergence holds. Then 
\[
	\sqrt N \Big(\bra{\mu^{(N)}_t, f } - \bra{\mu_t, f} \Big) \dto \bra{\xi, f},
\]
and it holds that 
\begin{equation}\label{limit-f}
	\bra{\xi, f} = \bra{\eta, f} + \int_0^t \bralr{\xi,  c \int \frac{f'(s, x) - f'(s, y)}{x - y}d\mu_s(y) + \frac12 f'' + \partial_s f}ds.
\end{equation}
\end{corollary}
\begin{proof}
	The formula~\eqref{Ito-f}, together with the above formulation, immediately yields the desired result.
\end{proof}

Next, we borrow an idea from \cite{Cabanal-French} to identify limiting processes more explicitly.
The idea is to choose orthogonal polynomials with respect to $\nu_c$. Let us extract here a useful result on orthogonal polynomials (see \cite[Chapter 2]{Deift-book-1999} or \cite[\S 3.8]{Simon}, for example). From two sequences $a_n \in \R$ and $b_n > 0$, a Jacobi matrix $J$ is formed by 
\[
	J = \begin{pmatrix}
		a_1	&b_1	\\
		b_1	&a_2		&b_2\\
		0	&b_2		&a_3		&b_3\\
		&&\ddots	&\ddots	&\ddots
	\end{pmatrix}.
\]
Then there is a probability measure $\mu$ such that 
\[
	\bra{\mu, x^n} = J^n(1,1), \quad n = 0, 1, \dots.
\]
In case there is a unique probability measure satisfying the above moment condition, or the measure $\mu$ is determined by moments, we call $\mu$ the spectral measure of $J$. The existence follows from Theorem 3.8.4 in \cite{Simon} while Corollary 3.8.4 therein provides a useful sufficient condition for the uniqueness, that is, the uniqueness holds, if $\sum_{n=1}^\infty b_n^{-1} = \infty$.

The spectral measure $\mu$ is related to orthogonal polynomials as follows. Define a sequence of polynomials $\{p_n\}$ as 
\begin{align*}
	&p_0 = 1, \quad p_1 = x - a_1,\\
	&p_{n+1} = xp_n - a_{n+1}p_n - b_n^2 p_{n-1}, \quad (n \ge 1).
\end{align*}
Then $\{p_n\}$ are orthogonal polynomials with respect to the spectral measure $\mu$,
\begin{equation}\label{orthogonal}
	\int p_m (x) p_n(x) d\mu(x) = \delta_{mn} \prod_{i=1}^n b_i^2,
\end{equation}
where $\delta_{mn}=1$, if $m=n$ and $\delta_{mn}=0$, if $m \neq n$.

Back to our problem, we need the fact that the probability measure $\nu_c$ is the spectral measure of the following Jacobi matrix \cite{DS15}
\[
	J = \begin{pmatrix}
		0	&\sqrt{c+1}\\
		\sqrt{c+1}	&0	&\sqrt{c+2}\\
			&\sqrt{c+2}	&0	&\sqrt{c+3}\\
		&&\ddots	&\ddots	&\ddots
	\end{pmatrix}.
\]
Then polynomials $\{p_n\}$ recursively defined as 
\begin{align*}
	&p_0 = 1, \quad p_1 = xp_0 = x, \\
	&p_{n+1} = xp_n - (n+c)p_{n-1}, \quad (n \ge 1),
\end{align*}
are orthogonal polynomials with respect to $\nu_c$, 
\begin{equation}\label{associated-Hermite}
	\bra{p_n, p_m}_{\nu_c} := \int p_n(x) p_m(x) d\nu_c(x) = \delta_{mn} \prod_{i=1}^n (c+i).
\end{equation}
When $c=0$, $\{p_n\}$ are nothing but probabilist's Hermite polynomials. For general $c$, polynomials $\{p_n\}$ are defined by shifting the coefficients in the recurrence relation by $c$, and thus, they are called associated Hermite polynomials.

Note that the polynomial $p_n$ is odd for odd $n$ and even for even $n$. Then we can choose a primitive $P_n$ of $p_n$, a polynomial of order $(n+1)$, to be either an odd or even polynomial. Define 
\begin{equation}\label{def-of-Pn}
	\tilde P_n(t, x) = t^{(n+1)/2} P_n(x/\sqrt t).
\end{equation}
Then $\tilde P_n(t,x)$ is a polynomial in $t$ and $x$. Now Corollary~\ref{lem:polytx} implies that 
\begin{align}
	&\sqrt N \Big(\bra{\mu^{(N)}_t, \tilde P_n(t, x) } - \bra{\mu_t, \tilde P_n(t, x)} \Big) \notag\\
	& \dto \bra{\xi, \tilde P_n} 
	= \bra{\eta, \tilde P_n} + \int_0^t \bralr{\xi,  c \int \frac{\tilde P_n'(s, x) - \tilde P_n'(s, y)}{x - y}d\mu_s(y) + \frac12 \tilde P_n'' + \partial_s \tilde P_n}ds.\label{limit-Pn}
\end{align}
As proved latter in Theorem~\ref{thm:Gaussian-CLT-orthogonal}, by working with $\tilde P_n$, the limiting processes $\{\bra{\eta, \tilde P_n}\}_{n \ge 0}$ are independent.

\begin{lemma}\label{lem:Gaussian-vanishing}
	For each $n$, it holds that 
\[
	c \int \frac{\tilde P'_n(s, x) - \tilde P'_n(s, y)}{x - y} d\mu_s(y) + \frac12 \tilde P''_n + \partial_s \tilde P_n = const\times s^{(n-1)/2}.
\]
Consequently, the last term in the equation \eqref{limit-Pn} vanishes.
\end{lemma}
\begin{proof}
It is clear that
\begin{align}
	&\tilde P_n'(s, x) = s^{n/2 } P_n'(x/\sqrt s) = s^{n/2 } p_n(x/\sqrt s),	\label{derivative-Pn}\\
	&\tilde P''_n(s, x) = s^{(n-1)/2} p'_n(x/\sqrt s), \notag\\
	&\partial_s \tilde P_n(s, x) = \frac{n+1}2 s^{(n-1)/2} P_n(x/\sqrt s) -\frac12 s^{(n-2)/2} x p_n(x/\sqrt s). \notag
\end{align}
Next, by using the property that $\int f(x/\sqrt s) d\mu_s(x) = \int f(x) d\nu_c(x) $, we obtain that 
\[
	c \int \frac{\tilde P'_n(s, x) - \tilde P'_n(s, y)}{x - y} d\mu_s(y) + \frac12 \tilde P''_n + \partial_s \tilde P_n = s^{(n-1)/2} F(x/\sqrt s),
\]
where
\begin{equation}\label{relation-p}
	F(x) = c \int \frac{p_n(x) - p_n(y)}{x - y} d\nu_c(y) + \frac12 p_n'(x) + \frac {n+1}2 P_n(x) - \frac12 xp_n(x).
\end{equation}
Thus, it suffices to show that $F(x) = const$.

Define a sequence $\{q_n\}$ as
\[
	q_n(x) = \int \frac{p_n(x) - p_n(y)}{x - y} d\nu_c(y).
\]
Then it is straightforward to check that $\{q_n\}$ are polynomials satisfying the same recurrence relation as $\{p_n\}$ but with different initial conditions
\[
	q_{n+1} = x q_n - (n+c) q_{n-1}, \quad (n \ge 1), \quad q_0 = 0, q_1 = 1.  
\]
Now the first term of $F(x)$ in the equation~\eqref{relation-p} is equal to $c q_n(x)$. To prove that $F(x) = const$, we show that its derivative is zero, that is, 
\begin{equation}\label{derivative}
	c q_n'(x) +  \frac12 p_n''(x) + \frac {n}2 p_n(x) - \frac12 x p_n'(x)  = 0.
\end{equation}
This relation can be proved by induction using the recurrence relation of $\{p_n\}$ and $\{q_n\}$. We omit detailed arguments but noting that in the proof, we need another relation  
\[
	p_n' + cq_n - (n+c)p_{n-1} = 0,
\]
which can be easily proved by induction. The lemma is proved.
\end{proof}

\begin{theorem}\label{thm:Gaussian-CLT-orthogonal}
For any $M\in \N$, 
$\{\sqrt N (\bra{\mu_t^{(N)}, \tilde P_{n}(t, x)} - \bra{\mu_t, \tilde P_{n}(t, x)}) \}_{n=0}^M$ jointly converge in distribution to $\{\tilde \eta_n(t)\}_{n=0}^M$, 
where $\{\tilde \eta_n(t)\}$ are independent Gaussian processes with mean zero and  covariance 
\begin{equation}\label{covariance}
	\Ex [\tilde\eta_m (s) \tilde \eta_n(t)] =\delta_{mn} \frac{\bra{p_n, p_n}_{\nu_c}}{n+1} (s \wedge t)^{n+1}.
\end{equation}
In particular, for the Gaussian beta ensemble~\eqref{GbE} (the distribution of the eigenvalue process $\{\lambda_i(t)\}$ at time $t = 1$), as $N \to \infty$ with $\beta N = 2c$,
\[
	\sqrt N \Big(\bra{L_{N}, P_{n}}  - \bra{\nu_c, P_{n}}\Big )\dto \Normal(0, \sigma_{P_{n}}^2) ,
\]
with 
\[
	\sigma^2_{P_n} = \frac{1}{n+1} \bra{p_n, p_n}_{\nu_c} = \frac{1}{n+1} (c+1) \cdots (c + n).
\]
For the joint convergence, they converge to independent Gaussian random variables.
\end{theorem}
\begin{proof}
	Corollary~\ref{lem:polytx} and Lemma~\ref{lem:Gaussian-vanishing} imply the joint convergence and the limiting processes $\bra{\xi, \tilde P_n}$ coincide with the limit of $ \Phi^{(\tilde P_{n}'; N)}$, which we denote by $\tilde \eta_n$.
Then Lemma~\ref{lem:joint} states that the limiting processes $\{\tilde \eta_n\}$ have mean zero and covariance 
\[
	\Ex[\tilde \eta_m(s) \tilde \eta_n(t)] =  \int_0^{s\wedge t} \bra{\mu_u, \tilde P_{m}'(u, x) \tilde P_{n}' (u, x)}du.
\]
It now follows from the formula~\eqref{derivative-Pn} for the derivatives of $\tilde P_m$ and $\tilde P_n$ that  
\[
	\bra{\mu_u, \tilde P_{m}'(u, x) \tilde P_{n}' (u, x)} = u^{(n+m)/2} \bra{\mu_u, p_m(x/\sqrt u) p_n(x/\sqrt u)} =  u^{(n+m)/2} \bra{p_m, p_n}_{\nu_c},
\]
which is zero when $n \neq m$. 
Thus, the covariance of $\tilde \eta_m(s)$ and $\tilde \eta_n(t)$ is given by the equation~\eqref{covariance}. That the covariance is zero when $m\neq n$ implies that Gaussian processes $\{\tilde \eta_m\}$ are independent, which completes the proof of  the first part of this theorem. The second one is just a particular case of the first. The proof is complete.
\end{proof}

\begin{remark}
Let $\tilde p_n = p_n / \sqrt{\bra{p_n, p_n}_{\nu_c}}$ be a sequence of orthonormal polynomials. For a polynomial $F$, we express $f = F'$ in terms of $\{\tilde p_n\}$,
\[
	f = \sum_{n} \bra{f, \tilde p_n}_{\nu_c} \tilde p_n = \sum_n \alpha_n \tilde p_n.
\]
We wish to simplify the limiting variance 
\[
	\sigma_F^2 =  \sum_{n} \frac{1}{n+1} \alpha_n^2 = \sum_{n} \frac{1}{n+1} \bra{f, \tilde p_n}_{\nu_c}^2.
\]
Let us now express the limiting variance as 
\begin{align*}
	\sigma_F^2 &= \sum_{n} \frac{1}{n+1} \bra{f, \tilde p_n}^2\\
	& = \sum_{n} \frac{1}{n+1} \int f(x)\tilde p_n(x) \nu_c(x) dx \int f(y)\tilde p_n(y) \nu_c(y) dy\\
	&= \iint f(x) f(y) \left(\sum_n    \frac{1}{n+1} \tilde p_n(x) \tilde p_n(y) \right) \nu_c(x)  \nu_c(y)  dx dy.
\end{align*}
Note that $\{\tilde p_n\}_{n=0}^\infty$ is an orthonormal basis in the space $L^2(\R, \nu_c)$. It follows that $\{\tilde p_n(x) \tilde p_n(y)\}_{n=0}^\infty$ is an orthonormal system in $L^2(\R^2, \nu_c \otimes \nu_c)$, and thus, 
\[
	K_c(x, y) = \sum_{n=0}^\infty  \frac{1}{n+1} \tilde p_n(x) \tilde p_n(y)
\]
is well-defined with equality in $L^2$. Now for any polynomial $F$, 
\[
	\sigma_F^2 = \iint F'(x) F'(y) K_c(x, y) \nu_c(x) \nu_c(y) dx dy.
\]

We are unable to simplify $K_c(x,y)$ yet. However, in case of Hermite  polynomials (the trivial case where $c = 0$), Mehler's formula gives us that for $\rho \in [0,1)$, 
\[
	 \sum_{n = 0}^\infty  \frac{p_n(x) p_n(y)}{n!} {\rho^n}= \frac{1}{\sqrt{1 - \rho^2}} \exp \left(- \frac{\rho^2(x^2 + y^2) - 2 \rho xy}{2 (1 - \rho^2)} \right) .
\]
Thus, 
\[
	K_0(x, y) = \int_0^1 \frac{1}{\sqrt{1 - \rho^2}} \exp \left(- \frac{\rho^2(x^2 + y^2) - 2 \rho xy}{2 (1 - \rho^2)} \right) d\rho.
\]
\end{remark}

\section{Laguerre case}
This section deals with the Laguerre case. We consider the following (scaled) beta Laguerre ensembles with parameters $\alpha, \beta >0$ whose joint density is proportional to
\begin{equation}\label{LbE}
	\prod_{i<j}|\lambda_j - \lambda_i|^{\beta} \prod_{l=1}^N \lambda_l^{\alpha-1} e^{-\lambda_l}, \quad (\lambda_1, \dots, \lambda_N) \in \Omega_L,
\end{equation}
where
\[
	\Omega_L = \{(\lambda_1, \dots, \lambda_N)  \in \R^N : 0 \le \lambda_1 \le \cdots \le \lambda_N\}.
\]
They are generalizations of the eigenvalue distribution of Wishart and Laguerre matrices in terms of the joint density. Refer to Chapter 1 in \cite{Forrester-book} for some basic properties of these ensembles.

The so-called beta Laguerre processes satisfy the following system of SDEs
\begin{equation}
\begin{cases}
	d\lambda_i(t)=\sqrt{2\lambda_i}  d  b_i(t)   + \alpha  d  t + \frac\beta2 \sum\limits_{j : j \neq i} \dfrac{2\lambda_i(t)}{\lambda_i(t) - \lambda_j(t)} dt,\\
	\lambda_i(0) = 0,
\end{cases}
 i = 1, \dots, N,
\end{equation}
and $(\lambda_1(t), \dots, \lambda_N(t)) \in \Omega_L$, for all $t > 0$,
where $\{b_i(t)\}_{i=1}^N$ are independent standard Brownian motions. For $\beta = 1$ and $\beta = 2$, they are the eigenvalue process of the Wishart process and Laguerre process, respectively \cite{Bru-1989,Bru-1991,Katori-Tanemura-2004,Konig-Oconnell-2001}. For $\beta \ge 1$, the above system of SDEs has a unique strong solution \cite{Graczyk-Jacek-2014}. However, when $\beta \in (0,1)$, we should define beta Laguerre processes as the squared of type B radial Dunkl processes \cite{Demni-2007-arxiv}. For this, the condition $\alpha > 1/2$ is required. Similar to the Gaussian case, the set of $t$ such that $\lambda_i(t) = \lambda_j(t)$, for some $i \neq j$, has Lebesgue measure zero, almost surely. In what follows, we consider a regime where $\alpha > 1/2$ is fixed and $\beta = 2c/N$ for a given constant $c \in (0, \infty)$. Under the above trivial initial condition ($\lambda_i(0) = 0$), for $t > 0$, the joint distribution of $\{\lambda_i(t) / t\}$ coincides with the beta Laguerre ensemble~\eqref{LbE} (cf.\ \cite{Rosler-Voit-1998}).

Let 
\[
	\mu^{(N)}_t = \frac1{N} \sum_{i = 1}^N \delta_{\lambda_i(t)}
\]
be the empirical measure process.
For $f(t, x) \in C^2([0, \infty) \times [0, \infty))$, by It\^o's formula, we obtain
\begin{align}
	d \bra{\mu^{(N)}_t, f} 
	&= \sum_{i = 1}^N \frac1N  \sqrt{2\lambda_i(t)} f'(t, \lambda_i(t)) db_i(t) \notag \\
	&\quad + \bra{\mu^{(N)}_t, \alpha f'(t, x)   + x f''(t, x) + \partial_t f(t, x)}dt \notag\\
	&\quad +  c \iint \frac{xf'(t, x) - yf'(t, y)} {x - y} d\mu^{(N)}_t(x) d\mu^{(N)}_t(y)dt  \notag\\
	&\quad  - \frac cN \bra{\mu^{(N)}_t, xf''(t, x) + f'(t, x)} dt. \label{L-Ito-for-f}
\end{align}
Again, note that the above formula holds when $\lambda_1(t), \dots, \lambda_N(t)$ are distinct.

We follow the same routine as in the Gaussian case. The equation~\eqref{L-Ito-for-f} with $f = x^n$ yields a recurrence relation for the $n$th moment process $S_n^{(N)}$,
\begin{align}
	d S_n^{(N)}(t)  &= \sum_{i = 1}^N \frac{n}{N} \sqrt{2\lambda_i(t)} \lambda_i(t)^{n-1} db_i(t)  \notag \\
	&\quad + \alpha n S_{n-1}^{(N)}(t) dt + c n \sum_{i = 0}^{n-1} S_i^{(N)}(t) S_{n-i-1}^{(N)}(t) dt \notag  \\
	&\quad + n(n-1) S_{n-1}^{(N)}(t) dt - \frac{c n^2}N S_{n-1}^{(N)}(t) dt. \label{L-moment}
\end{align}
Let
\[
	M_n^{(N)}(t) =\frac{n}{N} \sum_{i = 1}^N \int_0^t   \sqrt{2\lambda_i(s)} \lambda_i(s)^{n-1} db_i(s) = \frac{\sqrt 2 n}{N} \sum_{i = 1}^N \int_0^t    \lambda_i(s)^{n-1/2} db_i(s) 
\]
be the martingale part with the quadratic variation 
\begin{equation}\label{quadratic-of-M}
	[M_n^{(N)}](t) = \frac{2n^2}{N}   \int_0^t \frac{\sum_{i = 1}^N \lambda_i(s)^{2n-1}}{N}ds.
\end{equation}

Writing down $S_n^{(N)}(t)$ in the integral form
and using similar induction arguments as in the Gaussian case, we deduce that the $n$th moment process $S_n^{(N)}(t)$ converges to a deterministic limit $m_n(t)$ satisfying 
\begin{equation}\label{ODE-moment}
\begin{cases}
	\partial_t m_n(t) =  n\left( (\alpha + n - 1) m_{n - 1}(t) + c \sum_{i = 0}^{n - 1} m_i(t) m_{n - i - 1}(t) \right),\\
	m_n(0) = 0,
\end{cases}
\quad (n \ge 1).
\end{equation}
Again, these equations lead to a simple solution
\[
	m_n(t) = u_n t^n,
\]
where $\{u_n\}$ are defined recursively as
\[
	u_n =  (\alpha + n - 1) u_{n - 1} + c \sum_{i = 0}^{n - 1} u_i  u_{n - i - 1}, \quad n = 1,2,\dots, (u_0 = 1).
\]
The measure $\nu_{\alpha, c}$ having moments $\{u_n\}$ is unique and its density can be directly calculated from the above relation by a method in \cite{Martin-Kearney-2010}. However, we will use an existing result that $\nu_{\alpha, c}$ is the probability measure of associated Laguerre polynomials of Model II \cite{Trinh-Trinh-2021, Trinh-Trinh-BLP}. Its explicit density can be found in \cite{Ismail-et-al-1988}, or in \cite{Allez-Wishart-2013}.

The limiting probability measure-valued process $(\mu_t)$ is now defined as 
\[
\mu_t(dx) = \frac1{t}\nu_{\alpha, c}(x / t) dx, \quad(t >0),  \mu_0 = \delta_0.
\]
We obtain the following law of large numbers.
\begin{theorem}\label{thm:Laguerre-LLN}
For any polynomial $f(t, x)$ in $t$ and $x$, as $N \to \infty$,
	\[
		\bra{\mu_t^{(N)}, f} \Pto \bra{\mu_t, f}.
	\]
Recall that this is the convergence in probability in the space $C([0, T])$.
Moreover, it holds that 
\begin{equation}\label{integral-equation-L}
	\partial_t \bra{\mu_t, f} = c \iint \frac{xf'(t, x) - yf'(t, y)}{x - y}d\mu_t(x) d\mu_t(y)  + \bra{\mu_t,  \alpha f' + x f''+ \partial_t f}.
\end{equation}
\end{theorem}

Next, we study the central limit theorem.
Let 
\[
	\tilde S_n^{(N)}(t) = \sqrt N \left( S_n^{(N)}(t) - m_n(t) \right).
\]
Then by using the equations~\eqref{L-moment} and \eqref{ODE-moment}, we deduce that
\begin{align*}
	d \tilde S_n^{(N)}(t) &= d \sqrt N M_n^{(N)}(t)  + n (\alpha + n - 1) \tilde S_{n-1}^{(N)}(t) dt \\
	&\quad + cn \sum_{i=0}^{n-1} \left(\sqrt N S_i^{(N)}(t)S_{n-i-1}^{(N)}(t) - \sqrt N m_i(t) m_{n-i-1}(t) \right) dt\\
		&\quad  - \frac{c n^2}{\sqrt N} S_{n-1}^{(N)}(t) dt .
\end{align*}

We state without proof the following results analogous to the Gaussian case.
\begin{theorem}\label{thm:Laguerre-CLT}
The following hold.
\begin{itemize}
	\item[\rm(i)] For each $n\ge 1$, the martingale part $\sqrt N M_n^{(N)}$ converges in distribution to a Gaussian process $\eta_n$ of mean zero. The joint convergence also holds.
	
	\item[\rm(ii)] $\tilde S_n^{(N)}$ converges in distribution to a Gaussian process $\xi_n$ defined inductively as 
\[
	\xi_n(t) =  \eta_n(t) + n(\alpha + n - 1) \int_0^t \xi_{n-1}(s) ds + 2 cn \sum_{i=0}^{n-1} \left( \int_0^t m_i(s) \xi_{n-i-1}(s) ds \right),
\]
for $n \ge 2,$
with $\xi_0(t) \equiv 0, \xi_1(t) = \eta_1(t)$. The joint convergence also holds.
	
	\item [\rm(iii)] More generally, for any polynomial $f=f(t, x)$, 
	\begin{align*}
		\sqrt{N} \left(\bra{\mu_t^{(N)}, f} - \bra{\mu_t, f} \right) \dto \bra{\xi, f},
	\end{align*}
where the limiting Gaussian processes $\bra{\xi, f}$ satisfies the following relation 
\begin{align}\label{L-xi}
	\bra{\xi, f} &= \bra{\eta, f} \notag\\
	&\quad + \int_0^t \bralr{\xi,  2c \int \frac{x f'(s, x) - y f'(s, y)} {x - y} d\mu_s(y) + \partial_s f +  \alpha f'   + x f''} ds,
\end{align}
with $\bra{\eta, f}$ the limit of the martingale part 
\[
	\frac1{\sqrt N}\sum_{i = 1}^N  \int_0^t  \sqrt{2\lambda_i} f'(s, \lambda_i) db_i(s) \dto \bra{\eta, f}.
\]
Gaussian processes $\bra{\eta, f}$ have mean zero and are coupling with covariance 
\begin{equation}\label{L-Cov}
	\Ex[\bra{\eta, f}(s) \bra{\eta, g}(t)] = 2\int_0^{s\wedge t} \bra{\mu_u, xf'(u, x) g'(u, x)} du.
\end{equation}

\end{itemize}
\end{theorem}

Based on these results, especially based on the covariance formula~\eqref{L-Cov},  in order to study more about the limiting processes $\bra{\xi, f}$, the idea now is to take orthogonal polynomials with respect to $x\nu_{\alpha, c}(dx)$. For that purpose, let us extract some results on the probability measure $\nu_{\alpha, c}$. 
\begin{itemize}
	\item[(i)] The probability measure $\nu_{\alpha, c}$ is the spectral measure of the following Jacobi matrix 
\[
	J_{\alpha, c} = 
	\begin{pmatrix}
		c_1\\
		d_1	& c_2\\
		&\ddots	&\ddots
	\end{pmatrix}
	\begin{pmatrix}
		c_1 & d_1	\\
		&c_2		&d_2\\
		&&\ddots	&\ddots
	\end{pmatrix}
	=\begin{pmatrix}
		c_1^2 	&c_1 d_1	\\
		c_1 d_1	&c_2^2 + d_1^2		&c_2d_2\\
		&\ddots	&\ddots	&\ddots
	\end{pmatrix},
\]
where 
\[
	c_n = \sqrt{\alpha + c + n - 1},  \quad d_n = \sqrt{c + n}. 
\]

\item[(ii)] The explicit formula for the density $\nu_{\alpha, c}(x)$ is given by 
\[
	\nu_{\alpha, c}(x) = \frac{1}{\Gamma(c+1) \Gamma(c+\alpha)} \frac{x^{\alpha-1} e^{-x}}{|\Psi(c, 1-\alpha; x e^{-i \pi})|^2}, \quad x \ge 0.
\]
Here $\Psi(a, b; z)$ is Tricomi's confluent hypergeometric function.
This is Model II of associated Laguerre polynomials \cite{Ismail-et-al-1988}.

\item[(iii)] The probability measure $\tilde \nu_{\alpha, c}(dx) = \frac{1}{\alpha + c} x \nu_{\alpha, c}(x) dx$ is the probability measure of Model I of associated Laguerre polynomials, that is, $\tilde \nu_{\alpha, c}(dx)$ is the spectral measure of the following Jacobi matrix 
\begin{equation}\label{Jacobi-matrix-I}
	\tilde J_{\alpha, c} 
	=\begin{pmatrix}
		c_2^2 + d_0^2	&c_2 d_1	\\
		c_2 d_1	&c_2^2 + d_2^2		&c_3d_2\\
		&\ddots	&\ddots	&\ddots
	\end{pmatrix}.
\end{equation}
This observation is based on the explicit formulae for the two models of associated Laguerre polynomials in \cite{Ismail-et-al-1988}. We give another explanation in Remark~\ref{rem:associated} below.
\end{itemize}

From the Jacobi matrix $\tilde J_{\alpha, c}$, we define orthogonal polynomials with respect to $x\nu_{\alpha}$ or $\tilde \nu_{\alpha, c}$ as follows
\begin{align}
	&p_0 = 1, \quad p_1(x) = x - (c_2^2 + d_0^2) = x - (\alpha + 2c + 1),\notag\\
	&p_{n+1} = x p_n - (c_{n+1}^2 + d_{n+1}^2) p_n - c_{n+1}^2 d_n^2 p_{n-1}\notag \\
	&\qquad= x p_n - (\alpha + 2c + 2n + 1)p_n - (\alpha + c + n)(c + n)p_{n-1}, \quad (n \ge 1).\label{Laguerre-polynomials}
\end{align}
The orthogonal relations are given by
\[
	\int p_m (x) p_n(x) d\tilde \nu_{\alpha, c}(x) = \delta_{mn} \prod_{i=1}^n (c_{i+1}^2 d_i^2) = \delta_{mn}  \prod_{i=1}^n (c+i) (\alpha + c + i).
\]

Let $P_n$ be a primitive of $p_n$. Define  
\[
	\tilde P_n(t, x) := t^{n+1} P_n\left (\frac{x}{ t} \right).
\]
Again, the idea of working with $\tilde P_n(t,x)$ is to make the limiting processes $\{\bra{\eta, \tilde P_n}\}$ independent.
\begin{lemma}
	The last term in the equation~\eqref{L-xi} vanishes for $f = \tilde P_n(t, x)$.
\end{lemma}
\begin{proof}
A direct calculation shows that 
\begin{align*}
	&\tilde P_n'(t, x)  = t^{n} p_n(\frac xt),\\
	&\tilde P_n''(t, x)  = t^{n-1} p_n'(\frac xt),\\
	&\partial_t \tilde P_n(t, x) = (n+1) t^{n} P_n(\frac xt) - x t^{n-1} p_n(\frac xt).
\end{align*}
Thus, if suffices to show that 
\begin{equation}\label{L-p}
	2c \int \frac{x p_n(x) - y p_n(y)}{x - y} d\nu_{\alpha, c}(y) + \alpha p_n(x) + x p'_n(x) + (n+1) P_n(x) - x p_n(x) = const.
\end{equation}

We only sketch main steps in the proof of the above equation. 

Step 1. Define polynomials $\{q_n\}$ as
\[
	q_n := \int\frac{p_n(x) - p_n(y)}{x - y} \tilde \nu_{\alpha, c}(y) dy.
\]
Then $\{q_n\}$ satisfy the same recurrence relation as $\{p_n\}$ but with different initial conditions $q_0 = 0, q_1 = 1$. We now express the integral in the first term of the equation~\eqref{L-p} in terms of $p_n$ and $q_n$, 
\begin{align*}
	&\int \frac{x p_n(x) - y p_n(y)}{x - y} d\nu_{\alpha, c}(y) \\
	&\quad = \int \frac{xp_n(x) - yp_n(x) + yp_n(x)- yp_n(y)}{x - y} d\nu_{\alpha, c}(y) \\
	&\quad =p_n(x) + \int \frac{p_n(x) - p_n(y)}{x - y}(\alpha + c)d\tilde \nu_{\alpha, c}(y)\\
	&\quad =p_n(x) + (\alpha + c)q_n(x).
\end{align*}
This enables us to show the equation~\eqref{L-p} by induction.

Step 2. The equation~\eqref{L-p} is equivalent to the following obtained by taking its derivative 
\[
	(\alpha + 2c + 1)p_n' + 2c(\alpha + c)q_n' + xp_n'' + np_n - xp_n' = 0.
\] 
When showing that relation by induction, we need a further relation that
\[
	x p_n' - n p_n + c(c+\alpha)q_n - (\alpha + c + n)(c + n)p_{n-1} = 0
\]
which can be showed easily by induction. The proof is complete.
\end{proof}
We arrive at main results for the Laguerre case. Their detailed proofs are omitted because arguments are quite similar to the Gaussian case.
\begin{theorem}\label{thm:Laguerre-CLT-orthogonal}
	For any $M$, 
\[
	\Big\{\sqrt N \Big(\bra{\mu_t^{(N)}, \tilde P_n(t, x)} - \bra{\mu_t, \tilde P_n(t, x)} \Big)\Big\}_{n=0}^M
\]
jointly converge to independent Gaussian processes $\tilde \eta_n(t)$ of mean zero and  covariance
\begin{align*}
	\Ex[\tilde \eta_n (s) \tilde \eta_n(t)] &= \frac{(s\wedge t)^{2n + 2}}{n+1} \int p_n(x)^2 x d\mu_{\alpha, c}(x) \\
	&=  \frac{(s\wedge t)^{2n + 2} (\alpha + c)}{n+1} \int p_n(x)^2 d\tilde \mu_{\alpha, c}(x)\\
	&= \frac{(s\wedge t)^{2n + 2} (\alpha + c) }{n+1}\prod_{i=1}^n (c+i) (\alpha + c +i).
\end{align*}
\end{theorem}

\begin{corollary}
Consider the  beta Laguerre ensemble~\eqref{LbE} in a regime where $\alpha > \frac12$ is fixed and $\beta N = 2c$. Then the following hold.
\begin{itemize}
	\item[\rm(i)] The empirical distribution $L_N$ converges weakly to $\nu_{\alpha, c}$ as $N \to \infty$ (in probability).
	
	\item[\rm(ii)] As $N \to \infty$,
	\[
		\sqrt N \Big(\bra{L_N, P_n} - \bra{\nu_{\alpha, c}, P_n} \Big) \dto \Normal(0, \sigma_{P_n}^2),
	\]
where 
\[
	\sigma_{P_n}^2  =  \frac{\alpha + c}{n+1} \prod_{i=1}^n (c+i) (\alpha + c + i).
\]
The limits of the joint convergence are independent Gaussian random variables.
\end{itemize}
\end{corollary}

Note that by using a tridiagonal matrix model for beta Laguerre ensembles, the strong LLN (with the almost sure convergence) for $L_N$ and CLTs for continuously differentiable functions were established in \cite{Trinh-Trinh-2021}. 
We conclude this paper with the following remark.
\begin{remark}\label{rem:associated}We claim that if $\mu$ is the spectral measure of 
\[
J=	\begin{pmatrix}
		c_1\\
		d_1	& c_2\\
		&\ddots	&\ddots
	\end{pmatrix}
	\begin{pmatrix}
		c_1 & d_1	\\
		&c_2		&d_2\\
		&&\ddots	&\ddots
	\end{pmatrix}
	=\begin{pmatrix}
		c_1^2 	&c_1 d_1	\\
		c_1 d_1	&c_2^2 + d_1^2		&c_2d_2\\
		&\ddots	&\ddots	&\ddots
	\end{pmatrix},
\]
then $\nu = c_1^{-2} \times x \mu$ is the spectral measure of 
\[
		H=\begin{pmatrix}
		c_1^2 + d_1^2	&c_2 d_1	\\
		c_2 d_1	&c_2^2 + d_2^2		&c_3d_2\\
		&\ddots	&\ddots	&\ddots
	\end{pmatrix},
\]
provided that the spectral measure of $J$ (and of $H$) is unique.
Equivalently, given two Jacobi matrices $J$ and $H$ of the above form, the claim means that for any $n \ge 0$,
\begin{equation}\label{JH}
	c_1^2 H^{n}(1,1) = J^{n+1}(1,1).
\end{equation}

We roughly present an idea of proof. Since both sides of the relation~\eqref{JH} are polynomials of $\{c_i\}$ and $\{d_i\}$, we assume without loss of generality that the two sequences are bounded.
Consider an operator~~${}\hat{} \colon J \mapsto \hat J$
\[
J=	
\begin{pmatrix}
		c_1^2 	&c_1 d_1	\\
		c_1 d_1	&c_2^2 + d_1^2		&c_2d_2\\
		&\ddots	&\ddots	&\ddots
	\end{pmatrix}
\mapsto
	\hat J = \begin{pmatrix}
		0	&c_1\\
		c_1	&0	&d_1\\
		&d_1		&0	&c_2\\
		&&c_2	&0	&d_2\\
		&&&\ddots	&\ddots	&\ddots
	\end{pmatrix}.
\]
Then one may check the relation that $(\hat J)^{2n}(1,1) = J^n(1,1)$. In other words, if $\mu$ and $\hat \mu$ are the spectral measure of $J$ and $\hat J$, respectively,  then
\[
	\bra{\hat \mu, x^{2n}} = \bra{\mu, x^n},\quad \bra{\hat \mu, x^{2n+1}} = 0, \quad n = 0, 1, \dots.
\]
It follows that $S_{\hat \mu}(z) = z S_{\mu}(z^2)$, where $S_\nu(z)$ denotes the Stieltjes transform of a measure $\nu$, 
\[
	S_\nu(z) = \int \frac{\nu(dx)}{x - z}, \quad z \in \C \setminus \R.
\]
Formally, 
\[
	S_\nu(z) = - \sum_{n=0}^\infty \frac{\bra{\nu, x^n}}{z^{n+1}}
\]
is a generating function of the sequence of moments $\bra{\nu, x^n}$.

Consider Jacobi matrices 
\[
	K = \begin{pmatrix}
		d_0^2	&	d_0 c_1\\
		d_0 c_1	&c_1^2 + d_1^2	&d_1 c_2\\
		&d_1 c_2	&c_2^2 + d_2^2	&d_2 c_3	\\
		&&\ddots	&\ddots	&\ddots
	\end{pmatrix}, \hat K = \begin{pmatrix}
		0	&d_0\\
		d_0&0	&c_1\\
		&c_1	&0	&d_1\\
		&&d_1		&0	&c_2\\
		&&&c_2	&0	&d_2\\
		&&&&\ddots	&\ddots	&\ddots
	\end{pmatrix}.
\]
Then $\hat J$ is obtained from $\hat K$ by removing the top row and the left most column. Similarly, $H$ is obtained from $K$ in the same way. Let $\xi$ (resp.\ $\hat \xi$, $\mu$ and $\nu$) be the spectral measure of $K$ (resp.\ $\hat K$, $J$ and $H$). We need to show that $c_1^2 \nu = x \mu$. The following relations hold:
\begin{align*}
	&-\frac{1}{S_{\xi}(z)} = z - d_0^2 + d_0^2 c_1^2 S_{\nu}(z),\\
	&-\frac{1}{S_{\hat \xi}(z)} = z - 0 + d_0^2 S_{\hat \mu}(z),\\
	&S_{\hat \mu}(z) = z S_{\mu}(z^2), \\
	&S_{\hat \xi}(z) = z S_{\xi}(z^2).
\end{align*}
Here the first two relations are relations of the Stieltjes transform of the spectral measure of the original Jacobi matrix and that of the Jacobi matrix after removing  the top row and the left most column \cite[Theorem 3.2.4]{Simon}.
From which, we deduce that 
\[
	1 + z^2 S_\mu(z^2) = c_1^2 S_\nu(z^2),
\]
implying the desired relation~\eqref{JH}.

\end{remark}

\bigskip 
\noindent\textbf{Acknowledgements.}
The authors would like to thank the editor and the reviewers for valuable comments.
This work is supported by JSPS KAKENHI Grant Numbers 20K03659 (F.N.) and JP19K14547 (K.D.T.). Trinh Hoang Dung, ID VNU.2021.NCS.11, thanks The Development Foundation of Vietnam National University, Hanoi for sponsoring this research.

\begin{footnotesize}


\end{footnotesize}
\end{document}